\documentclass[a4paper,12pt]{article}
\usepackage[top=2.5cm,bottom=2.5cm,left=2.5cm,right=2.5cm]{geometry}
\usepackage{cite,amsmath,amssymb}
    \usepackage[margin=1cm,%
                font=small,%
                format=hang,%
                labelsep=period,%
                labelfont=bf]{caption}
\usepackage[algoruled,linesnumbered]{algorithm2e}
\usepackage{algorithmic}

\usepackage{amsfonts,epsf,here,graphicx}
\usepackage{latexsym,multirow,rotating}
\usepackage{pgf,tikz,subfigure}
\usepackage{epstopdf}

\newtheorem{theorem}{\bf Theorem}[section]
\newtheorem{corollary}[theorem]{\bf Corollary}
\newtheorem{lemma}[theorem]{\bf Lemma}

\newcommand{\qed}{\hfill $\square$ \bigskip}

\begin{document}

\baselineskip=0.30in
\vspace*{40mm}

\begin{center}
{\LARGE \bf New methods for calculating the degree distance and the Gutman index}
\bigskip \bigskip

{\large \bf Simon Brezovnik, \qquad
Niko Tratnik
}
\bigskip\bigskip

\baselineskip=0.20in

\textit{Faculty of Natural Sciences and Mathematics, University of Maribor, Slovenia} \\
{\tt simon.brezovnik2@um.si, niko.tratnik@um.si}
\medskip

\bigskip\medskip

(Received \today)

\end{center}

\noindent
\begin{center} {\bf Abstract} \end{center}
In the paper we develop new methods for calculating the two well-known topological indices, the degree-distance and the Gutman index. Firstly, we prove that the Wiener index of a double vertex-weighted graph can be computed from the Wiener indices of weighted quotient graphs with respect to a partition of the edge set that is coarser than $\Theta^*$-partition. This result immediately gives a method for computing the degree-distance of any graph. Next, we express the degree-distance and the Gutman index of an arbitrary phenylene by using its hexagonal squeeze and inner dual. In addition, it is shown how these two indices of a phenylene can be obtained from the four quotient trees. Furthermore, reduction theorems for the Wiener index of a double vertex-weighted graph are presented. Finally, a formula for computing the Gutman index of a partial Hamming graph is obtained.
\vspace{3mm}\noindent


\baselineskip=0.30in



\section{Introduction}

In this paper, we study methods for calculating the degree distance and the Gutman index of a graph. These graph invariants are important degree- and distance-based topological indices. The \textit{degree distance} of a connected graph $G$, denoted by $DD(G)$, is defined as
$$DD(G) = \sum_{\lbrace u,v \rbrace \subseteq V(G)}(\textrm{deg}(u) + \textrm{deg}(v))d(u,v)$$
and was firstly introduced in 1994 \cite{dobrynin}. However, a similar concept was invented five years earlier by H. Schultz (see \cite{schultz}) and therefore, the degree distance is sometimes referred to as the Schultz index. The \textit{Gutman index} of a connected graph $G$, denoted by $Gut(G)$, was introduced in \cite{gutman_index} and it is defined in the following way:
 $$Gut(G) = \sum_{\lbrace u,v \rbrace \subseteq V(G)}\textrm{deg}(u)\textrm{deg}(v)d(u,v).$$

\noindent 
Both topological indices are special variations of the well known Wiener index, which was introduced in 1947 by H. Wiener \cite{wiener} as
$$W(G)= \sum_{\lbrace u,v \rbrace \subseteq V(G)}d(u,v).$$
In fact, the Gutman index and the degree distance are special cases of weighted Wiener indices. Therefore, our results are stated in a general way. 

The degree-distance and the Gutman index were extensively studied in the past. For example, some recent research on the degree-distance can be found in \cite{hua,khalifeh-2009,zhu} and on the Gutman index in \cite{knor,mazo}. Moreover, both indices were studied and compared in \cite{das,ghala,guo,gutman1}. In addition, paper \cite{khalifeh-2010} provides some methods for computing these two indices on partial cubes. Furthermore, recently some generalizations using the Steiner distance were introduced \cite{gutman,mao,mao1}.

The paper is organized in six sections. The following section states some basic definitions and preliminary results. 

We continue with a cut method for computing the degree-distance of a graph. Note that a cut method is a powerful tool for investigating distance-based topological indices \cite{klavzar-2015}. In \cite{glavni} a cut method for computing the degree-distance of a benzenoid system was developed. Therefore, in section 3 we generalize this result such that it can be applied for any connected graph by using quotient graphs with respect to a partition coarser than $\Theta^*$-partition. In addition, our result holds for the Wiener index of a double vertex-weighted graph and is analogous to Theorem \ref{rac_wie} from \cite{nad_klav}, which can be applied to calculate the Wiener index of a weighted graph. In section 4, the mentioned theorems are used for computing the degree-distance and the Gutman index of phenylenes. More precisely, these indices of an arbitrary phenylene are expressed by using its hexagonal squeeze and inner dual.  Such result is already known for the Wiener index of a phenylene \cite{klav_gut-1997}. Moreover, it is shown how these two indices of a phenylene can be obtained from the four quotient trees.

In \cite{redukcija} some reduction theorems for calculating the Wiener index of a weighted graph were proposed. In section 5, we apply a similar method to obtain such theorems also for the Wiener index of a double vertex-weighted graph, which can then be used also for the degree-distance. As an example, we calculate the degree-distance and the Gutman index of an infinite family of graphs that were firstly considered in \cite{nad_klav}. 

Finally, in section 6 we show an inequality for the Wiener index of a weighted graph and apply it for the Gutman index. In particular, the equality holds if and only if the graph is a partial Hamming graph. Note that similar results are known for the Wiener index \cite{klavzar-2006} and for the degree-distance \cite{glavni}.

\section{Preliminaries}

All the graphs considered in this paper are simple and finite. For a graph $G$, the set of all the vertices is denoted by $V(G)$ and the set of edges by $E(G)$. Moreover, we define $d_G(u,v)$ to be the usual shortest-path distance between vertices $u, v\in V(G)$ and also sometimes write $d(u,v)$ instead of $d_G(u,v)$ if there is no confusion. Furthermore, for any $u \in V(G)$ the {\it open neighbourhood} $N(u)$ is defined as the set of all the vertices that are adjacent to $u$. The \textit{degree} of $u$, denoted by $\textrm{deg}(u)$, is defined as the cardinality of the set $N(u)$. Finally, we set $N[u]=N(u) \cup \lbrace u \rbrace$.
\smallskip

\noindent
If $G$ is a graph, we say that a function $w: V(G) \rightarrow \mathbb{R}^+$ is a \textit{weight}. A pair $(G,w)$ is then called a \textit{vertex-weighted graph} or shortly a \textit{weighted graph}. Moreover, if $a,b$ are weights, then the triple $(G,a,b)$ is called a \textit{double vertex-weighted graph}. 
\smallskip

\noindent
The Wiener index of a connected weighted graph $(G,w)$ can be defined in two different ways (see \cite{khalifeh-2010}):
$$W_{*}(G,w) = \sum_{\lbrace u,v \rbrace \subseteq V(G)}w(u)w(v)d(u,v),$$
$$W_{+}(G,w) = \sum_{\lbrace u,v \rbrace \subseteq V(G)}(w(u) + w(v))d(u,v).$$
However, since the first definition is more common, we usually write $W(G,w)$ instead of $W_{*}(G,w)$ and call it the Wiener index of a weighted graph. Obviously, the Gutman index and the degree distance are just special cases of these two weighted indices.
\smallskip

\noindent
The Wiener index of a connected double vertex-weighted graph $(G,a,b)$, introduced in \cite{glavni}, is defined as:
$$W(G,a,b)= \sum_{\lbrace u,v \rbrace \subseteq V(G)}\left(a(u)b(v) + a(v)b(u) \right)d(u,v).$$
Obviously, $W_{+}(G,a)$ is exactly $W(G,a,b)$ if we take $b\equiv 1$.
\smallskip

\noindent
Two edges $e_1 = u_1 v_1$ and $e_2 = u_2 v_2$ of graph $G$ are in relation $\Theta$, $e_1 \Theta e_2$, if
$$d(u_1,u_2) + d(v_1,v_2) \neq d(u_1,v_2) + d(u_1,v_2).$$
Note that this relation is also known as Djokovi\' c-Winkler relation.
The relation $\Theta$ is reflexive and symmetric, but not necessarily transitive.
We denote its transitive closure (i.e.\ the smallest transitive relation containing $\Theta$) by $\Theta^*$. It is easy to observe that any two edges in an odd cycle are in $\Theta^*$ relation. Moreover, any two diametrically opposite edges in an even cycle are in $\Theta$ relation.

Let $ \mathcal{E} = \lbrace E_1, \ldots, E_k \rbrace$ be the $\Theta^*$-partition of the set $E(G)$. Then we say that a partition $\lbrace F_1, \ldots, F_r \rbrace$ of $E(G)$ is \textit{coarser} than $\mathcal{E}$
if each set $F_i$ is the union of one or more $\Theta^*$-classes of $G$.
\smallskip

\noindent
Suppose $G$ is a graph and $F \subseteq E(G)$. The \textit{quotient graph} $G / F$ is a graph whose vertices are connected components of the graph $G \setminus F$, such that two components $C_1$ and $C_2$ are adjacent in $G / F$ if some vertex in $C_1$ is
adjacent to a vertex of $C_2$ in $G$.
\smallskip

\noindent
The following theorem gives a method for computing the Wiener index of a weighted graph.

\begin{theorem} \cite{nad_klav}
\label{rac_wie}
Let $G$ be a connected graph. If $\lbrace F_1, \ldots, F_k \rbrace$ is a partition coarser than the $\Theta^*$-partition, then
$$W(G,w) = \sum_{i=1}^k W(G / F_i, w_i),$$
where $w_i: V(G/F_i) \rightarrow \mathbb{R}^+$ is defined by $w_i(C)= \sum_{x \in V(C)} w(x)$ for all connected components $C$ of the graph $G \setminus F_i$.
\end{theorem}

\noindent
The \textit{Cartesian product} $G_1 \Box \cdots \Box G_k$ of graphs $G_1, \ldots, G_k$ has the vertex set $V(G_1)\times \cdots \times V(G_k)$, two vertices $(u_1,\ldots,u_k)$ and $(v_1,\ldots, v_k)$ being adjacent if they differ in exactly
one position, say in $i$th, and $u_iv_i$ is an edge of $G_i$. 
\smallskip

\noindent
Let $H$ and $G$ be two graphs. A function $\ell: V(H) \rightarrow V(G)$ is called an \textit{embedding of $H$ into $G$} if $\ell$ is injective and, for any two vertices $u,v \in V(H)$, $\ell(u)\ell(v) \in E(G)$ if $uv \in E(H)$. If such a function $\ell$ exists, we say that $H$ can be \textit{embedded} in $G$. An embedding $\ell$ of graph $H$ into graph $G$ is called an \textit{isometric embedding} if for any two vertices $u,v \in V(H)$ it holds $d_H(u,v) = d_G(\ell(u),\ell(v))$. Moreover, subgraph $H$ of a graph $G$ is called an \textit{isometric subgraph} if for each $u,v \in V(H)$ it holds $d_H(u,v) = d_G(u,v)$.
\smallskip

\noindent
A \textit{Hamming graph} is the Cartesian product of complete graphs and a \textit{partial Hamming graph} is any isometric subgraph of a Hamming graph. In
the particular case where all the factors are $K_2$'s we speak of hypercubes and partial cubes,
respectively. Partial cubes constitute a large class of graphs with a lot of applications and includes, for example, many families of chemical graphs (benzenoid systems, trees, phenylenes, cyclic phenylenes, polyphenyls). Partial Hamming graphs and in particular partial cubes were studied and characterized in many papers (for example, see \cite{klavzar-book}).
\smallskip

\noindent
For an edge $ab$ of a graph $G$, let $W_{ab}$ be the set of vertices of $G$ that are closer to $a$ than
to $b$. We write $\langle S \rangle$ for the subgraph of $G$ induced by $S \subseteq V(G)$. Moreover, a subgraph $H$ of $G$ is called {\it convex} if for arbitrary vertices $u,v \in V(H)$ every shortest path between $u$ and $v$ in $G$ is also contained in $H$. The following theorem proved by Djokovi\' c and Winkler puts forth two fundamental characterizations of partial cubes:
\begin{theorem} \cite{klavzar-book} \label{th:partial-k} For a connected graph $G$, the following statements are equivalent:
\begin{itemize}
\item [(i)] $G$ is a partial cube.
\item [(ii)] $G$ is bipartite, and $\langle W_{ab} \rangle $ and $\langle W_{ba} \rangle$ are convex subgraphs of $G$ for all $ab \in E(G)$.
\item [(iii)] $G$ is bipartite and $\Theta = \Theta^*$.
\end{itemize}
\end{theorem}
Is it also known that if $G$ is a partial cube and $E$ is a $\Theta$-class of $G$, then $G \setminus E$ has exactly two connected components, namely $\langle W_{ab} \rangle $ and $\langle W_{ba} \rangle$, where $ab \in E$. For more information about partial cubes see \cite{klavzar-book}.
\smallskip

\noindent
The \textit{canonical embedding} of a connected graph $G$ is defined as follows: Let the $\Theta^*$-classes of $G$ be $E_1,E_2,\ldots,E_k$. Let $\alpha_i : V(G)\rightarrow V(G / E_i)$ be the map sending any $u \in V(G)$ to the component of $G \setminus E_i$ that contains it. The canonical embedding
$\alpha : V(G) \rightarrow G / E_1 \Box \cdots \Box G / E_k$
is defined by
$$\alpha(u) = (\alpha_1(u), \ldots, \alpha_k(u))$$
for any $u \in V(G)$. It is well known that the canonical embedding is irredundant isometric embedding, see \cite{klavzar-book}. Here irredundant means that every factor graph $G / E_i$, $i \in \lbrace 1,\ldots,k \rbrace $, has at least two vertices and that each vertex of $G / E_i$ appears as a coordinate of some vertex $\alpha(u)$, where $u \in V(G)$.


\baselineskip=16pt

\section{The degree distance via unification of $\Theta^*$-classes}

The main goal of this section is to find a formula for calculating the degree distance of a graph. However, we present a result in a more general setting, i.e.\,by proving it for the Wiener index of a double vertex-weighted graph. Some chemical applications are given in the next section.

Let $G$ be a connected graph and $\lbrace F_1, \ldots, F_r \rbrace$ a partition coarser than the $\Theta^*$-partition. For any $u \in V(G)$ and $i \in \lbrace 1, \ldots, r \rbrace$ we denote by $\ell_i(u)$ the connected component of the graph $G \setminus F_i$ which contains $u$. The result of the following lemma was proved in \cite{nad_klav}. A complete proof can also be found in \cite{tratnik}.

\begin{lemma} \cite{nad_klav,tratnik} \label{distance}
Let $G$ be a connected graph. If $\lbrace F_1, \ldots, F_r \rbrace$ is a partition coarser than the $\Theta^*$-partition, then for any $u,v \in V(G)$ it holds
$$d_G(u,v) = \sum_{i=1}^r d_{G / F_i}(\ell_i(u),\ell_i(v)).$$
\end{lemma}

\noindent
The following theorem is the main result of the section.

\begin{theorem} \label{splosen}
If $(G,a,b)$ is a connected double vertex-weighted graph with a partition $\{F_1,...,F_r\}$ coarser than $\Theta^*$-partition, 
then
$$W(G,a,b)=\sum_{i=1}^r W(G/F_i, a_i, b_i),$$
where $a_i: V(G/F_i) \rightarrow \mathbb{R}^+$ is defined by $a_i(C)= \sum_{x \in V(C)} a(x)$ and $b_i: V(G/F_i) \rightarrow \mathbb{R}^+$ is defined by $b_i(C)= \sum_{x \in V(C)} b(x)$ for all connected components $C$ of the graph $G \setminus F_i$.
\end{theorem}

\begin{proof}
By the definition and Lemma \ref{distance} we obtain
\begin{eqnarray*}
W(G,a,b) & = & \sum_{\{u,v\} \in V(G)} (a(u)b(v)+a(v)b(u))d_G(u,v) \\
& = & \sum_{\{u,v\} \in V(G)}(a(u)b(v)+a(v)b(u))\left( \sum_{i=1}^r d_{G/F_i}(\ell_i(u),\ell_i(v))\right) \\ 
& = & \sum_{i=1}^r \left(\sum_{\{u,v\} \in V(G)} (a(u)b(v)+a(v)b(u)) d_{G/F_i}(\ell_i(u),\ell_i(v)) \right).
\end{eqnarray*}

\noindent
Let $X,Y$ be two arbitrary distinct connected components of $G \setminus F_i$. Obviously, for any $x,x' \in V(X)$ and $y, y' \in V(Y)$ it holds $d_{G/F_i}(\ell_i(x),\ell_i(y))=d_{G/F_i}(\ell_i(x'),\ell_i(y'))$. Moreover,

\begin{eqnarray*}
\sum_{x \in V(X)} \sum_{y \in V(Y)} \left(a(x)b(y)+a(y)b(x) \right)& = & \sum_{x \in V(X)}a(x)\sum_{y \in V(Y)}b(y) + \sum_{y \in V(Y)}a(y)\sum_{x \in V(X)}b(x) \\
& = & a_i(X)b_i(Y) + a_i(Y)b_i(X). \\
\end{eqnarray*}

\noindent
Therefore,
\begin{eqnarray*}
W(G,a,b)& = & \sum_{i=1}^r \left(\sum_{\{X,Y\} \in V(G/F_i)} (a_i(X)b_i(Y) + a_i(Y)b_i(X)) d_{G/F_i}(X,Y) \right) \\
& = & \sum_{i=1}^r W(G/F_i, a_i, b_i), \\
\end{eqnarray*}

\noindent
which finishes the proof. \qed

\end{proof}

In the rest of the section we present some consequences of the previous theorem. Firstly, the obtained result can be used to compute $W_+(G,a)$ for a connected weighted graph.

\begin{corollary}
Let $(G,a)$ be a connected weighted graph with a partition $\{F_1,...,F_r\}$ coarser than $\Theta^*$-partition.
It holds
$$W_+(G,a)=\sum_{i=1}^r W(G/F_i, a_i,b_i)$$
where $a_i: V(G/F_i) \rightarrow \mathbb{R}$ is defined by $a_i(C)= \sum_{x \in V(C)} a(x)$ and $b_i: V(G/F_i) \rightarrow \mathbb{R}^+$ is defined by $b_i(C)= |V(C)|$ for all connected components $C$ of the graph $G \setminus F_i$.
\end{corollary}

\begin{proof}
If we set $b\equiv1$ in Theorem \ref{splosen}, we quickly get the obtained formula. \qed

%
%

\noindent

\end{proof}

\noindent
Finally, we are able to calculate the degree distance of an arbitrary graph $G$.
\begin{corollary} \label{rac_degree}
If $G$ is a connected graph with a partition $\{F_1,...,F_r\}$ coarser than $\Theta^*$-partition, 
then
$$DD(G)=\sum_{i=1}^r W(G/F_i, a_i, b_i),$$
where $a_i: V(G/F_i) \rightarrow \mathbb{R}$ is defined by $a_i(C)= \sum_{x \in V(C)} {\rm deg}(x)$ and $b_i: V(G/F_i) \rightarrow \mathbb{R}^+$ is defined by $b_i(C)= |V(C)|$ for all connected components $C$ of the graph $G \setminus F_i$.
\end{corollary}

If $G$ is a partial cube and $E_1, \ldots, E_k$ are the $\Theta$-classes of $G$, we know that the graph $G \setminus E_i$ has exactly two connected components for any $i \in \lbrace 1,\ldots, k \rbrace$ and these two components will be denoted by $C_i^1$ and $C_i^2$. Moreover, for any $i \in \lbrace 1, \ldots, k \rbrace$ and $j \in \lbrace 1,2 \rbrace$ we define $A_j(E_i) = \sum_{x \in V(C_i^j)} a(x)$ and $B_j(E_i) = \sum_{x \in V(C_i^j)} b(x)$. The next result follows directly from Theorem \ref{splosen} and generalizes Lemma 4.2 from \cite{glavni} where it was stated just for trees.

\begin{corollary} \label{rac_delne}
If $(G,a,b)$ is a double vertex-weighted partial cube and if $E_1,\ldots,E_k$ are the $\Theta$-classes of $G$, then
$$W(G,a,b)=\sum_{i=1}^k \left( A_1(E_i)B_2(E_i) + A_2(E_i)B_1(E_i) \right).$$
\end{corollary}

\section{Applications to phenylenes}

In this section, we apply Corollary \ref{rac_degree} and Theorem \ref{rac_wie} to obtain the relationships between the degree-distance and the Gutman index of a phenylene with the weighted Wiener indices of its hexagonal squeeze and the inner dual. Moreover, it is described how the two indices can be computed from four weighted quotient trees which enables us to compute the indices in linear time.

Firstly, we need to introduce some additional notation. Let ${\cal H}$ be the hexagonal (graphite) lattice and let $Z$ be a cricuit on it. Then a {\em benzenoid system} is induced by the vertices and edges of ${\cal H}$, lying on $Z$ and in its interior. Let $B$ be a benzenoid system. A vertex shared by three hexagons of $B$ is called an \textit{internal} vertex of $B$. A benzenoid system is said to be \textit{catacondensed} if it does not possess internal vertices. Otherwise it is called \textit{pericondensed}. Two distinct hexagons with a common edge are called \textit{adjacent}. The \textit{inner dual} of a benzenoid system $B$, denoted as $ID(B)$, is a graph which has hexagons of $B$ as vertices, two being adjacent whenever  the corresponding hexagons are adjacent. Obviously, the inner dual of a catacondensed benzenoid system is always a tree.

Let $B$ be a catacondensed benzenoid system. If we add rectangles between all pairs of adjacent hexagons of $B$, the obtained graph $G$ is called a \textit{phenylene}. We then say that $B$ is a \textit{hexagonal squeeze} of $G$ and denote it by $HS(G)=B$, see Figure \ref{ben_phe}.

\begin{figure}[h] 
\begin{center}
\includegraphics[scale=0.7]{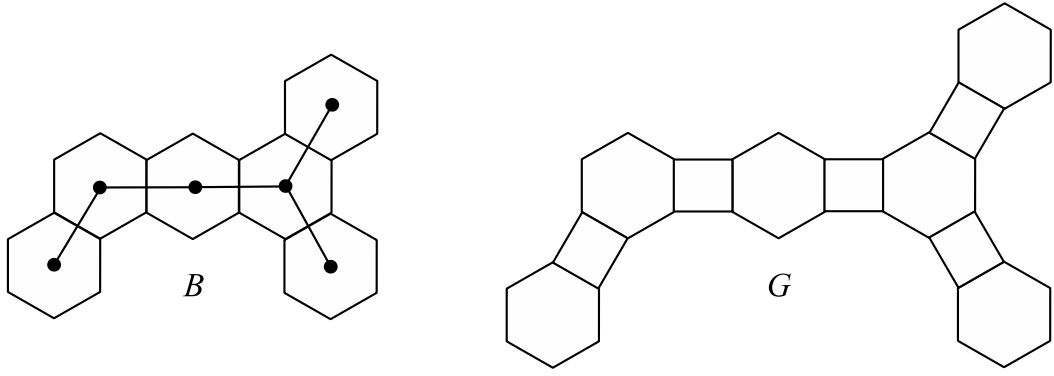}
\end{center}
\caption{\label{ben_phe} A benzenoid system $B$ with its inner dual and a phenylene $G$ such that $B=HS(G)$.}
\end{figure}

By Theorem \ref{th:partial-k} it follows that any benzenoid system or a phenylene is a partial cube. An \textit{elementary cut} $C$ of a benzenoid system or a phenylene $G$ is a line segment that starts at
the center of a peripheral edge of $G$,
goes orthogonal to it and ends at the first next peripheral
edge of $G$. By $C$ we usually also denote the set of edges that are intersected by the corresponding elementary cut. Elementary cuts in benzenoid systems have been
described and illustrated by numerous examples in several
earlier articles. The main insight for our consideration
is that every $\Theta$-class of a benzenoid system or a phenylene  coincides with exactly one of its elementary cuts. 

Let $G$ be a phenylene, $B=HS(G)$, and $T=ID(B)$. Firstly, we introduce four different weights in the following way. The weights $\omega_1,\omega_2 : V(B) \rightarrow \mathbb{R}^+$ are defined as 
$$w_1(u)= 4{\rm deg}(u) - 6,$$
$$w_2(u)= {\rm deg}(u) - 1$$
for any $u \in V(B)$. Moreover, the weights $\omega_3,\omega_4 : V(T) \rightarrow \mathbb{R}^+$ are defined as
 $$w_3(x)= 2{\rm deg}(x) + 12,$$
$$w_4(x)= 6$$
for any $x \in V(T)$.

\begin{theorem} If $G$ is a phenylene, $B=HS(G)$, and $T=ID(B)$, then it holds
$$DD(G)= W(B,w_1,w_2) + W(T,w_3,w_4),$$
$$Gut(G)= W(B,w_1)+W(T,w_3).$$
\end{theorem}

\begin{proof}
Let $F_1$ be the set of all the edges of $G$ that correspond to the edges of the hexagonal squeeze $B$ (the edges of all the hexagons of $G$). Moreover, let $F_2 = E(G) \setminus F_1$. Since the sets $F_1,F_2$ are both unions of elementary cuts of $G$, it is obvious that the partition $\lbrace F_1, F_2 \rbrace$ is a partition coarser than $\Theta$-partition. Therefore, by Corollary \ref{rac_degree} we have
$$DD(G) = W(G / F_1, a_1,b_1) + W(G / F_2, a_2, b_2),$$
\noindent
where $a_i$, $i \in \lbrace 1,2 \rbrace$, represents the sum of all the degrees in the corresponding connected components of the graph $G \setminus F_i$, and $b_i$, $i \in \lbrace 1,2 \rbrace$, represents the number of vertices in the corresponding connected components of the graph $G \setminus F_i$.

We also notice that $G / F_1 \cong B$ and $G / F_2 \cong T$. Furthermore, the connected components of $G \setminus F_1$ are either vertices (the vertices of degree 2 in $G$, which correspond to vertices of degree 2 in $B$) or edges (composed of two vertices of degree 3 in $G$, which correspond to vertices of degree 3 in $B$). Therefore,
$$a_1(u) = \begin{cases}
6; & {\rm deg}(u) = 3\\
2; & {\rm deg}(u) = 2
\end{cases}, \quad 
b_1(u)=\begin{cases}
2; & {\rm deg}(u) = 3\\
1; & {\rm deg}(u) = 2
\end{cases}$$
for any $u \in V(B)$. Obviously, $a_1 \equiv w_1$ and $b_1 \equiv w_2$.

On the other hand, the connected components of $G \setminus F_2$ are the hexagons of $G$, each of them corresponds to exactly one vertex from $T$. Therefore,
$$a_2(x) = \begin{cases}
18; & {\rm deg}(x) = 3\\
16; & {\rm deg}(x) = 2\\
14; & {\rm deg}(x) = 1
\end{cases}, \quad 
b_2(x)=6
$$
for any $x \in V(T)$. Obviously, $a_2 \equiv w_3$ and $b_2 \equiv w_4$, which completes the proof for the degree distance.

For the Gutman index we consider the same partition $\lbrace F_1, F_2 \rbrace$ of the set $E(G)$. By Theorem \ref{rac_wie} we obtain
$$Gut(G) = W(G / F_1, a_1) + W(G / F_2, a_2),$$
\noindent
where $a_i$, $i \in \lbrace 1,2 \rbrace$, represents the sum of all the degrees in the corresponding connected components of the graph $G \setminus F_i$. Since $a_1 \equiv w_1$ and $a_2 \equiv w_3$, the proof is complete. \qed
\end{proof}

Another way to compute the degree distance and the Gutman index of a phenylene is by using four weighted quotient trees, which are defined in the following way.

Let $G$ be a phenylene and $B$ the hexagonal squeeze of $G$. The edge set of $B$ can be naturally partitioned into sets $E_1'$, $E_2'$, and $E_3'$ of edges of the same direction. Denote the sets of edges of $G$ corresponding to the edges in $E_1'$, $E_2'$, and $E_3'$ by $E_1, E_2$, and $E_3$, respectively. Moreover, let $E_4 = E(G) \setminus (E_1 \cup E_2 \cup E_3)$ be the set of all the edges of $G$ that do not belong to $B$. The quotient graph $T_i$, $1\le i\le 4$, is then defined in the standard way as the graph $G / E_i$. In a similar way we can define the quotient graphs $T_1', T_2', T_3'$ of the hexagonal squeeze $B$. It is known  that for any benzenoid system its quotient graphs are trees, see \cite{chepoi-1996}. Obviously, a tree $T_i'$ is isomorphic to $T_i$ for $i=1,2,3$ and $T_4$ is isomorphic to the inner dual of $B$. Therefore, all quotient graphs $T_1,T_2,T_3,T_4$ are trees.

Now we extend a quotient tree $T_i$, $i \in \lbrace 1,2,3,4 \rbrace$, to a double vertex-weighted tree $(T_i,a_i,b_i)$ as follows: 
\begin{itemize}
\item for $C \in V(T_i)$, let $a_i(C)$ be the sum of all the degrees of vertices in the connected component $C$ of $G \setminus E_i$;
\item for $C \in V(T_i)$, let $b_i(C)$ be the number of vertices in the component $C$ of $G \setminus E_i$.
\end{itemize}

\noindent
Everything is prepared for the following theorem.

\begin{theorem} If $G$ is a phenylene and $(T_i,a_i,b_i)$, $i \in \lbrace 1,2,3,4 \rbrace$, are the corresponding double vertex-weighted quotient trees, then it holds
$$DD(G)= W(T_1,a_1,b_1) + W(T_2,a_2,b_2) + W(T_3,a_3,b_3) + W(T_4,a_4,b_4),$$
$$Gut(G)= W(T_1,a_1) + W(T_2,a_2) + W(T_3,a_3) + W(T_4,a_4).$$
\end{theorem}

\begin{proof}
It is clear the each set $E_i$, $i \in \lbrace 1,2,3,4 \rbrace$, is the union of some elementary cuts of $G$. Therefore, the partition $\lbrace E_1, E_2, E_3, E_4 \rbrace$ is a partition coarser than $\Theta$-partition. The theorem now follows by Corollary \ref{rac_degree} and Theorem \ref{rac_wie}. \qed
\end{proof}

By the same reasoning as in \cite{chepoi-1996} we can show that the double vertex-weighted quotient trees of a phenylene can be computed in linear time with respect to the number of vertices. Moreover, it was shown in \cite{chepoi-1997,glavni} that the Wiener index of any weighted tree or double vertex-weighted tree can be computed in linear time with respect to the number of vertices of a tree. Therefore, we easily obtain the following corollary.

\begin{corollary}
If $G$ is a phenylene with $n$ vertices, then the degree distance of $G$ and the Gutman index of $G$ can be computed in $O(n)$ time.
\end{corollary}

For an example, we consider phenylene $G$ from Figure \ref{ben_phe}. The double vertex-weighted quotient trees for $G$ are depicted in Figure \ref{Wiener_trees}. 

\begin{figure}[h] 
\begin{center}
\includegraphics[scale=0.6]{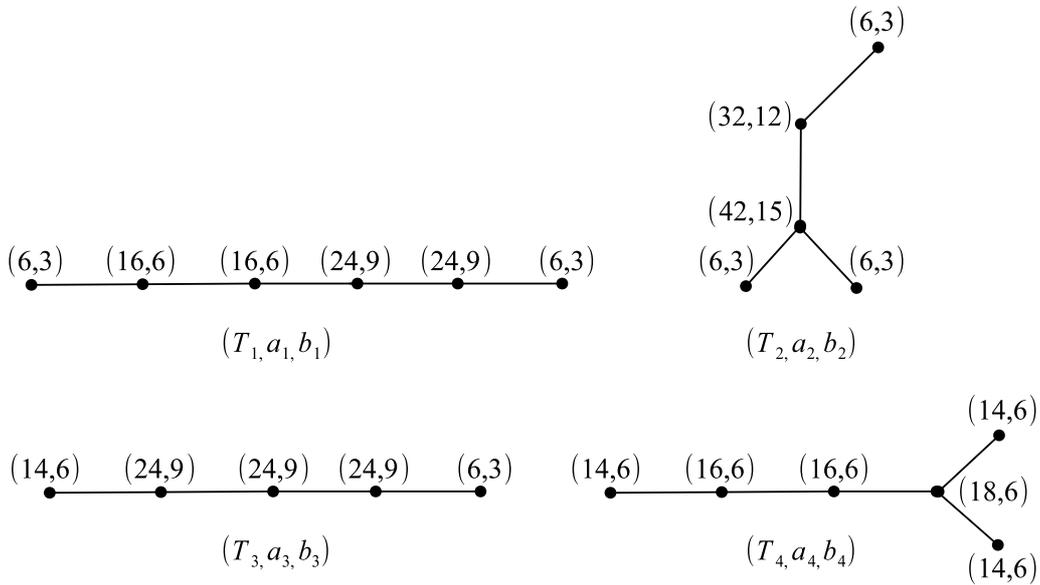}
\end{center}
\caption{\label{Wiener_trees} The double vertex-weighted quotient trees of phenylene $G$.}
\end{figure}

\noindent
We next compute the quantities (note that this can be done by using Corollary \ref{rac_delne}):
$$W(T_1, a_1,b_1) = 5208, W(T_2,a_2,b_2) = 2976, W(T_3,a_3,b_3) = 4416, W(T_4,a_4,b_4) = 5784.$$

\noindent
Therefore, the degree distance of $G$ is the sum

$$DD(G)=5208+2976+4416+5784=18384.$$

\noindent
Analogously, to compute the Gutman index we first calculate:
$$W(T_1, a_1) = 6484, W(T_2,a_2) = 3600, W(T_3,a_3) = 5520, W(T_4,a_4) = 7252.$$

\noindent
Finally, the Gutman index of $G$ is the sum

$$Gut(G)=6484+3600+5520+7252=22856.$$

\section{Reduction theorems for connected double vertex-weighted graphs}

We prove results analogous to the result from \cite{redukcija} for calculating the weighted Wiener index by a special reduction. To state the theorems, some additional definitions are needed.

\noindent
If $G$ is a graph, then vertices $x$ and $y$ are in relation $R$ if $N(x) = N(y)$. Obviously, $R$ is an equivalence
relation on $V(G)$. The $R$-equivalence class containing $x$ will be denoted with $[x]_R$. 

\noindent
Let $G$ be a connected graph,
$c \in V (G)$ and $C = [c]_R$. We define a new graph $G'$ with $G' = G \setminus (C \setminus \{c \})$. For any weight $w: V(G) \rightarrow \mathbb{R}^+$ we define $w': V(G') \rightarrow \mathbb{R}^+$ in the following way: $\omega'(c)= \sum_{x \in C}\omega(x)$ and $\omega'(x) = w(x)$
for any $x \notin C$. The next theorem was recently obtained.

\begin{theorem} \cite{redukcija} \label{posebna_redukcija0}
Let $(G,w)$ be a connected weighted graph, $c \in V(G)$, and $C=[c]_R = \{c_1,\ldots,c_k\}$. Then $$W(G,w)= W(G',w') + \sum_{\{c_i,c_j\} \subseteq C}2w(c_i)w(c_j).$$
\end{theorem}

\noindent
Finally, we can state the main theorem of the section.

\begin{theorem} \label{posebna_redukcija}
Let $(G,a,b)$ be a connected double vertex-weighted graph, $c \in V(G)$, and $C=[c]_R = \{c_1,\ldots,c_k\}$. Then $$W(G,a,b)= W(G',a',b') + \sum_{\{c_i,c_j\} \subseteq C}2(a(c_i)b(c_j)+a(c_j)b(c_i)).$$
\end{theorem}

\begin{proof}
If $|C|=1$, $(G',\omega')=(G,\omega)$, hence the result is trivial. Let  $c_1=c$ and $k \geq 2$. Then, we obtain
	
	\begin{itemize}
		\item [$(i)$] $d_G(c_i,x)=d_G(c_j,x)$ holds for any $c_i, c_j \in C$ and $x \notin C$,
			\item [$(ii)$] $d_G(x,y)=d_{G'}(x,y)$ holds for any vertices $x,y \in V(G) \setminus C$,
		\item [$(iii)$] $d_G(c_i,c_j)=2$ holds for any $c_i, c_j \in C$, $i \neq j$.

	\end{itemize}

	Using these facts we can compute the Wiener index of $(G,a,b)$ as follows:

\begin{eqnarray*}
W(G,a,b)&=& \sum_{\{x,y\} \subseteq V(G)}(a(x)b(y)+a(y)b(x))d_G(x,y) \\
&=& \sum_{x \notin C}\sum_{i=1}^k \left( a(c_i)b(x) + a(x)b(c_i) \right) d_G(c_i,x)  \\
&+& \sum_{\{x,y\} \subseteq V(G)\setminus C}(a(x)b(y)+a(y)b(x))d_G(x,y)  \\
&+& \sum_{\{c_i,c_j\} \subseteq C}(a(c_i)b(c_j)+a(c_j)b(c_i))d_G(c_i,c_j)  \\
&=& \sum_{x \notin C} \left( b(x)\sum_{i=1}^k a(c_i)d_G(c_i,x)+ a(x)\sum_{i=1}^k b(c_i)d_G(c_i,x) \right)  \\
&+& \sum_{\{x,y\} \subseteq V(G)\setminus C}(a(x)b(y)+a(y)b(x))d_G(x,y)  \\
&+& \sum_{\{c_i,c_j\} \subseteq C}(a(c_i)b(c_j)+a(c_j)b(c_i))d_G(c_i,c_j). 
\end{eqnarray*}
Therefore, applying properties $(i)$, $(ii)$, and $(iii)$ we deduce
\begin{eqnarray*}
W(G,a,b) &=& \sum_{x \in V(G') \setminus \lbrace c \rbrace} \left( b'(x)a'(c) + a'(x)b'(c) \right) d_{G'}(c,x) + \\
&+& \sum_{\{x,y\} \subseteq V(G')\setminus \lbrace c \rbrace}(a'(x)b'(y)+a'(y)b'(x))d_{G'}(x,y) + \\
&+& \sum_{\{c_i,c_j\} \subseteq C}2(a(c_i)b(c_j)+a(c_j)b(c_i)) \\
& = & W(G',a',b') + \sum_{\{a_i,a_j\} \subseteq C}2(a(a_i)b(a_j)+a(a_j)b(a_i))
\end{eqnarray*}
and the proof is complete. \qed
\end{proof}

\begin{corollary} \label{posledica_redukcije}
Let $(G,a,b)$ be a connected double vertex-weighted graph and let $a(c_i)=k_1$ and $b(c_i)=k_2$ for all $c_i \in C$, $k_1, k_2 \in \mathbb{R^+}$.
Then $$W(G,a,b)= W(G',a',b') + 2k_1k_2|C|(|C|-1).$$
\end{corollary}

In a similar way as before, we can introduce relation $S$ in the following way, see \cite{redukcija}. If $G$ is a graph, then vertices $x$ and $y$ are in relation $S$ if $N_G[x] = N_G[y]$. Obviously, $S$ is an equivalence
relation on $V(G)$ and the $S$-equivalence class containing $x$ will be denoted with $[x]_S$. Analogously as before we can define $(G',a',b')$ for any double vertex-weighted graph $(G,a,b)$ and $c \in V(G)$. The arguments in the proof of the next theorem are parallel to the proof of Theorem \ref{posebna_redukcija},
the only difference is that $d_G(c_i, c_j )$ equals one whenever $c_i,c_j$ are two distinct vertices from $C=[c]_S$.

\begin{theorem} 
Let $(G,a,b)$ be a connected double vertex-weighted graph, $c \in V(G)$, and $C=[c]_S = \{c_1,\ldots,c_k\}$. Then $$W(G,a,b)= W(G',a',b') + \sum_{\{c_i,c_j\} \subseteq C}(a(c_i)b(c_j)+a(c_j)b(c_i)).$$
\end{theorem}

In the rest of the section we use the obtained results on a family of graphs $G_n$, $n \geq 2$, where $n$ is the number of vertical layers, see Figure \ref{gn}. These graphs where introduced in \cite{nad_klav}, where the Wiener index was computed. Later, the obtained result was corrected in \cite{redukcija}. First, we determine the $\Theta^*$-classes of $G_n$, which are denoted by $E_1, \ldots, E_{n-1}$, see Figure \ref{gn}.

\begin{figure}[h] 
\begin{center}
\includegraphics[scale=0.6]{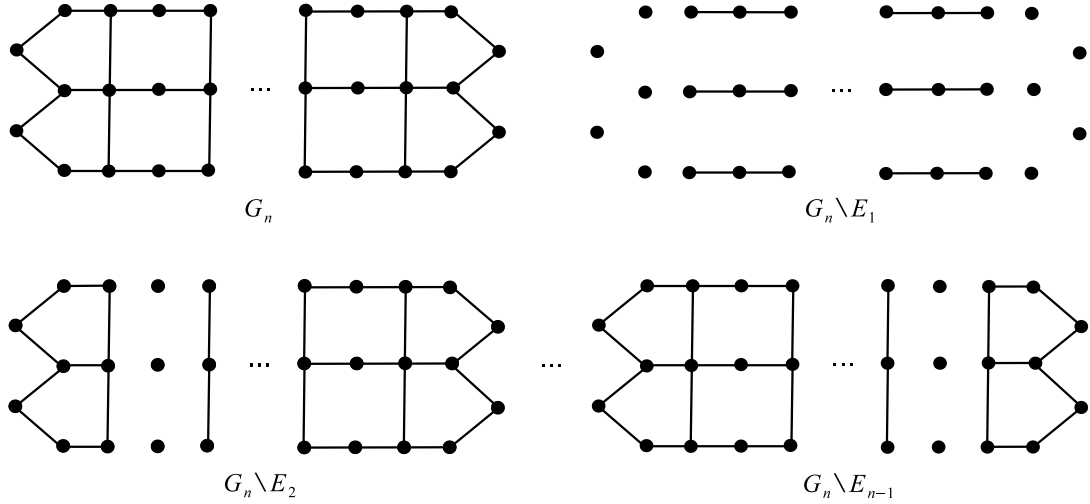}
\end{center}
\caption{\label{gn} Graph $G_n$, $n \geq 2$, and the subgraphs $G_n \setminus E_i$, $i \in \lbrace 1, \ldots, n-1 \rbrace$.}
\end{figure}

\noindent
Moreover, let $F_1=E_1$ and $F_2 = \bigcup_{i=2}^{n-1}E_i$. The graph $G_n \setminus F_2$ is depicted in Figure \ref{gn_zdruzeni_razred}. In addition, Figures \ref{kvocientniE1} and \ref{kvocientniE2} show quotient graphs $(G_n / F_1,a_1,b_1)$ and $(G_n/F_2,a_2,b_2)$, where the weights are defined as in Corollary \ref{rac_degree}.

\begin{figure}[h] 
\begin{center}
\includegraphics[scale=0.8]{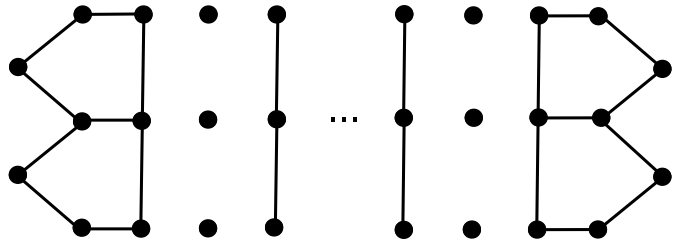}
\end{center}
\caption{\label{gn_zdruzeni_razred} The graph $G_n \setminus F_2$, where $F_2 = \bigcup_{i=2}^{n-1}E_i$.}
\end{figure}

\begin{figure}[h] 
\begin{center}
\includegraphics[scale=0.65]{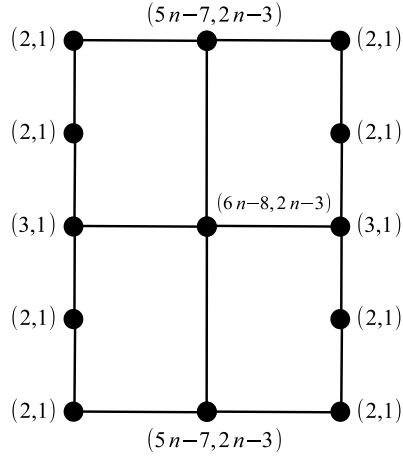}
\end{center}
\caption{\label{kvocientniE1} The graph $(G_n / F_1,a_1,b_1)$.}
\end{figure}

\begin{figure}[h] 
\begin{center}
\includegraphics[scale=0.65]{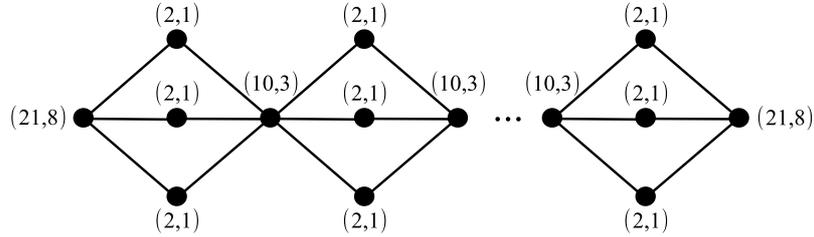}
\end{center}
\caption{\label{kvocientniE2} The graph $(G_n / F_2,a_2,b_2)$, where $F_2 = \bigcup_{i=2}^{n-1}E_i$.}
\end{figure}

\noindent
The Wiener indices of graphs $(G_n / F_1,a_1,b_1)$ and $(G_n / F_1,a_1)$ can be computed directly. The results are

$$W(G/F_1,a_1,b_1) = 84 n^2 + 354n -152,$$ 
$$W(G/F_1,a_1) = 110 n^2 + 404n -196.$$ 

\noindent
To simplify the calculation of the Wiener indices of graphs $(G_n / F_2,a_2,b_2)$ and $(G_n / F_2,a_2)$, we use Corollary \ref{posledica_redukcije} exactly $(n-2)$-times. Thus, we get the double vertex-weighted path on  $(2n-3)$ vertices and we denote it by $(P,a^{(n-2)},b^{(n-2)})$, see Figure \ref{redukcijaE2}.

\begin{figure}[h] 
\begin{center}
\includegraphics[scale=0.65]{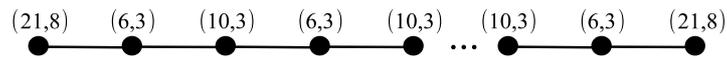}
\end{center}
\caption{\label{redukcijaE2} Path $(P,a^{(n-2)},b^{(n-2)})$ obtained by performing special reductions.}
\end{figure}

\noindent
To compute $W(P,a^{(n-2)},b^{(n-2)})$, we take into account separately the contributions of pairs of vertices such that
\begin{itemize}
\item both vertices have weights $(6,3)$,
\item both vertices have weights $(10,3)$,
\item one vertex has weights $(6,3)$ and the other $(10,3)$,
\item one vertex has weights $(21,8)$ and the other $(6,3)$ or $(10,3)$ or $(21,8)$.

\end{itemize}

\noindent
Therefore, we get the next result

\begin{eqnarray*}
W(P,a^{(n-2)},b^{(n-2)}) & = & (6 \cdot 3 + 3\cdot 6)\sum_{i=1}^{n-3} \sum_{j=1}^{i} (2j) + (10 \cdot 3 + 3\cdot 10)\sum_{i=1}^{n-4} \sum_{j=1}^{i} (2j) \\
& + & 2 \cdot (6 \cdot 3 + 3\cdot 10)\sum_{i=1}^{n-3} \sum_{j=1}^{i} (2j-1) \\
& + & 2 \cdot (21 \cdot 3 + 8 \cdot 6)\sum_{j=1}^{n-2}(2j-1) + 2 \cdot (21 \cdot 3 + 8 \cdot 10)\sum_{j=1}^{n-3}(2j) \\
& + & (21 \cdot 8 + 8 \cdot 21)(2n-4) \\
& = & 64n^3 + 16n^2 - 402n + 228.
\end{eqnarray*}

\noindent
Hence, by Corollary \ref{posledica_redukcije} we obtain
$$ W(G_n/F_2,a_2,b_2) = W(P,a^{(n-2)},b^{(n-2)}) + 24(n-2) = 64n^3 + 16n^2 - 378n + 180.$$

\noindent
Finally, by Corollary \ref{rac_degree} one can calculate
$$DD(G_n) = W(G_n/F_1,a_1,b_1) + W(G_n/F_2,a_2,b_2) = 64n^3 + 100n^2-24n+28.$$

Similarly, one can quickly get the Gutman index of $G_n$. We obtain

\begin{eqnarray*}
W(P,a^{(n-2)}) & = & 6^2 \sum_{i=1}^{n-3} \sum_{j=1}^{i} (2j) + 10^2\sum_{i=1}^{n-4} \sum_{j=1}^{i} (2j) \\
& + & 2 \cdot (6 \cdot 10)\sum_{i=1}^{n-3} \sum_{j=1}^{i} (2j-1) \\
& + & 2 \cdot (21 \cdot 6)\sum_{j=1}^{n-2}(2j-1) + 2 \cdot (21 \cdot 10)\sum_{j=1}^{n-3}(2j) \\
& + & 21^2 \cdot (2n-4) \\
& = & \frac{2}{3} \cdot(128n^3-731n+438).
\end{eqnarray*}

\noindent
Hence, by Theorem \ref{posebna_redukcija0} we get
$$ W(G_n/F_2,a_2) = W(P,a^{(n-2)}) + 24 \cdot (n-2) = \frac{2}{3} \cdot (128n^3-695n+366).$$

\noindent
Lastly, by Theorem \ref{rac_wie} we obtain
$$Gut(G_n) = W(G_n/F_1,a_1) + W(G_n/F_2,a_2) = \frac{2}{3} \cdot (128n^3+165n^2-89n +72).$$

\section{The Wiener index of weighted partial Hamming graphs}

In this section, we show how the Wiener index of a weighted partial Hamming graph can be easily computed. Naturally, the obtained result can be applied also for the Gutman index.


The following well known result will be needed and it was already used in \cite{klavzar-2006}. For the sake of completeness, we give the proof anyway.

\begin{lemma} \label{hamming_lema}
If $G$ is a connected graph, then $G$ is a partial Hamming graph if and only if all the quotient graphs with respect to $\Theta^*$-classes of $G$ are complete.
\end{lemma}

\begin{proof}
The backward implication is obvious and it follows from the definition of the canonical isometric embedding.

Let $G$ be a partial Hamming graph and let $\gamma : G \rightarrow \prod_{j=1}^m {K_{i_j}}$ be an isometric embedding into the Cartesian product of complete graphs $K_{i_j}$. Discard all the factors which contain only one vertex and the unused vertices in each factor. We get an irredundant embedding of a graph $G$ into a product of complete graphs. Therefore, by Theorem 5.3 from \cite{klavzar1993}, this embedding is the canonical isometric embedding, so the quotient graphs of $G$ with respect to $\Theta^*$-classes of $G$ are all complete.
\qed
\end{proof}

We are ready to prove the main theorem of this section, which generalizes a similar result from \cite{klavzar-2006}. In the rest of the section, the $\Theta^*$-classes of a connected graph $G$ will be denoted by $E_1, \ldots, E_k$, $k \in \mathbb{N}$. Moreover, for any $i \in \lbrace 1,\ldots, k \rbrace$, we denote the connected components of the graph $G \setminus E_i$ by $C_{i}^1, \ldots, C_{i}^{r_i}$.

\begin{theorem} \label{hamming_graf_izrek}
If $(G,w)$ is a connected weighted graph, then
$$W(G,w) \geq \frac{1}{2}\sum_{i=1}^k \sum_{j=1}^{r_i} w(C_i^j)w^c(C_i^j),$$
where $w(C_i^j) = \sum_{x \in V(C_i^j)}w(x)$ and $w^c(C_i^j) = \sum_{x \in V(G) \setminus V(C_i^j)}w(x)$ for any $i \in \lbrace 1,\ldots,k \rbrace, j \in \lbrace 1, \ldots, r_i\rbrace$. Moreover, the equality holds if and only if $G$ is a partial Hamming graph.
\end{theorem}

\begin{proof}
Let $\alpha = (\alpha_1, \ldots, \alpha_k)$ be the canonical isometric embedding for graph $G$. By the definition and Lemma \ref{distance} we have
\begin{eqnarray*}
W(G,w) & = & \sum_{\lbrace u,v \rbrace \subseteq V(G)} w(u)w(v)d_G(u,v) \\
& = & \sum_{\lbrace u,v \rbrace \subseteq V(G)} w(u)w(v) \left( \sum_{i=1}^k d_{G / E_i}(\alpha_i(u),\alpha_i(v)) \right) \\
& = & \sum_{i=1}^k \sum_{\lbrace u,v \rbrace \subseteq V(G)} w(u)w(v)d_{G / E_i}(\alpha_i(u),\alpha_i(v)).
\end{eqnarray*}

\noindent
In the rest of the proof, we introduce notation $$w^c(u) =\sum_{\substack{x \in V(G) \\ \alpha_i(u) \neq \alpha_i(x)}}w(x)$$ for any $u \in V(G)$. Since $d_{G / E_i}(\alpha_i(u),\alpha_i(v)) \geq 1$ for all $u,v \in V(G)$ with $\alpha_i(u) \neq\alpha_i(v)$, $i \in \lbrace 1,\ldots, k \rbrace$, we deduce
\begin{eqnarray*}
W(G,w) & \geq & \sum_{i=1}^k \sum_{\substack{\lbrace u,v \rbrace \subseteq V(G) \\ \alpha_i(u) \neq \alpha_i(v)}} w(u)w(v) \\
& = & \frac{1}{2}\sum_{i=1}^k \sum_{u \in V(G)} w(u)w^c(u) \\
& = & \frac{1}{2}\sum_{i=1}^k \sum_{j=1}^{r_i} \sum_{u \in C_i^j} w(u)w^c(u).
\end{eqnarray*}

\noindent
Obviously, for any $u,v \in V(C_i^j)$ it holds $w^c(u)=w^c(v)$. Hence,
\begin{eqnarray*}
W(G,w) & \geq & \frac{1}{2}\sum_{i=1}^k \sum_{j=1}^{r_i} w(C_i^j)w^c(C_i^j).
\end{eqnarray*}

Obviously, the equality holds if and only if $d_{G / E_i}(\alpha_i(u),\alpha_i(v)) = 1$ for all $u,v \in V(G)$ with $\alpha_i(u) \neq\alpha_i(v)$, $i \in \lbrace 1,\ldots, k \rbrace$, which is satisfied when all the quotient graphs with respect to $\Theta^*$ relation are complete graphs. By Lemma \ref{hamming_lema}, this is fulfilled if and only if $G$ is a partial Hamming graph. 
\qed

\end{proof}

\noindent
It is worth mentioning that the previous result can also be derived from Theorem \ref{rac_wie} and Lemma \ref{hamming_lema}.

\begin{corollary} \label{hamming_graf_posl}
If $G$ is a connected graph, then
$$Gut(G) \geq \frac{1}{2}\sum_{i=1}^k \sum_{j=1}^{r_i} {\rm deg}(C_i^j){\rm deg}^c(C_i^j),$$
where ${\rm deg}(C_i^j) = \sum_{x \in V(C_i^j)}{\rm deg}(x)$ and ${\rm deg}^c(C_i^j) = \sum_{x \in V(G) \setminus V(C_i^j)}{\rm deg}(x)$ for any $i \in \lbrace 1,\ldots,k \rbrace, j \in \lbrace 1, \ldots, r_i\rbrace$. Moreover, the equality holds if and only if $G$ is a partial Hamming graph.
\end{corollary}

For an example, consider the family of graphs $H_n, n \geq 2$, where $n$ denotes the number of inner faces, see Figure \ref{hiska}. It is clear that $|V(H_n)|= 2n+1$ holds for any $n \geq 2$. Figure \ref{hiska} also illustrates all $n$ graphs $H_n$ with removed $\Theta^*$-classes and in Figure \ref{quotientHn} we can see their quotient graphs.

\begin{figure}[h] 
\begin{center}
\includegraphics[scale=0.8]{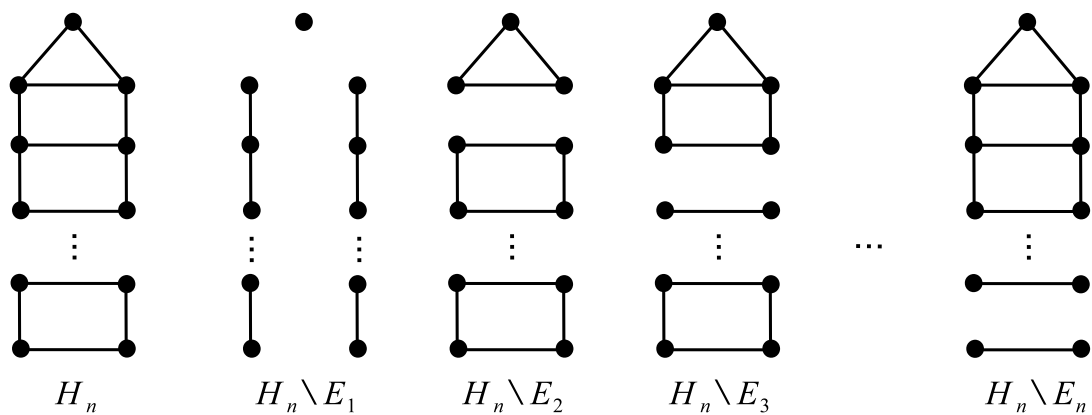}
\end{center}
\caption{\label{hiska} Graph $H_n$ and its subgraphs $H_n \setminus E_i, i \in \{ 1,2,...,n\}$.}
\end{figure}


\begin{figure}[h] 
\begin{center}
\includegraphics[scale=0.8]{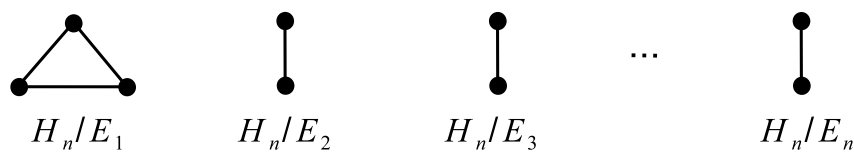}
\end{center}
\caption{\label{quotientHn} Quotient graphs $H_n / E_i, i \in \{ 1,2,...,n\}$.}
\end{figure}

All the quotient graphs with respect to $\Theta^*$ relation are complete, and therefore, by Lemma \ref{hamming_lema}, $H_n$ is a partial Hamming graph for any $n \geq 2$. Moreover, using Corollary \ref{hamming_graf_posl}, we can calculate the closed formula for the Gutman index of $H_n$ in the following way.

\begin{eqnarray*}
Gut(H_n) & = & \frac{1}{2}\sum_{i=1}^n \sum_{j=1}^{r_i} {\rm deg}(C_i^j){\rm deg}^c(C_i^j)\\
& = & \frac{1}{2} \big[ (3n-1)(3n+1) + (3n-1)(3n+1) + 2(6n-2) \big]\\
& + &  \sum_{i=1}^{n-1} \big[ (2+6i)(6(n-i-1)+4) \big] \\
& = & 6n^3 + 9n^2 - 4n + 1.
\end{eqnarray*}

%

\section*{Acknowledgment} 

\noindent The author Niko Tratnik acknowledge the financial support from the Slovenian Research Agency (research core funding No. P1-0297 and J1-9109).

\end{document}